\documentclass[a4paper,11pt] {amsart}
\usepackage{amsthm,amssymb,latexsym,mathrsfs}
\usepackage{color}

\input amssym.def

\author{Beno\^it F. Sehba and Edgar Tchoundja}

\title[Duality and Hankel operators on Bergman-Orlicz]{Duality for large Bergman-Orlicz spaces and Boundedness of Hankel Operators}
\newtheorem{theorem}{T{\hskip 0pt\footnotesize\bf HEOREM}}[section]
\newtheorem{lemma}[theorem]{L{\hskip 0pt\footnotesize\bf EMMA}}
\newtheorem{proposition}[theorem]{P{\hskip 0pt\footnotesize\bf ROPOSITION}}
\newtheorem{definition}[theorem]{D{\hskip 0pt\footnotesize\bf EFINITION}}
\newtheorem{corollary}[theorem]{C{\hskip 0pt\footnotesize\bf OROLLARY}}
\newtheorem{remark}[theorem]{R{\hskip 0pt\footnotesize\bf EMARK}}

\newcommand{\bprop} {\begin{proposition}}
\newcommand{\eprop} {\end{proposition}}
\newcommand{\btheo} {\begin{theorem}}
\newcommand{\etheo} {\end{theorem}}
\newcommand{\blem} {\begin{lemma}}
\newcommand{\elem} {\end{lemma}}
\newcommand{\bcor} {\begin{corollary}}
\newcommand{\ecor} {\end{corollary}}

\newcommand{\Be}{\begin{equation}}
\newcommand{\Ee}{\end{equation}}
\newcommand{\Bea}{\begin{eqnarray}}
\newcommand{\Eea}{\end{eqnarray}}
\newcommand{\Bes}{\begin{equation*}}
\newcommand{\Ees}{\end{equation*}}
\newcommand{\Beas}{\begin{eqnarray*}}
\newcommand{\Eeas}{\end{eqnarray*}}
\newcommand{\Ba}{\begin{array}}
\newcommand{\Ea}{\end{array}}

\def\C{\mathbb{C}}

\scrollmode

\begin{document}
\date{\today}
\address{Beno\^it Sehba, Department of Mathematics, University of Ghana, Legon\\ P. O. Box LG 62 Legon-Accra, Ghana.}
\email{bfsehba@ug.edu.gh}
\email{bsehba@gmail.com}
\address{Edgar Tchoundja, Department of Mathematics, University of Yaound\'e I, P. O. Box 812  Yaound\'e, Cameroon.}
\email{etchound@ictp.it}
\email{tchoundjaedgar@yahoo.fr}

\begin{abstract} For $\mathbb B^n$ the unit ball of $\mathbb C^n$, we consider Bergman-Orlicz spaces of holomorphic functions in $L^\Phi_\alpha$, which are generalizations of classical Bergman spaces. We characterize the dual space of large Bergman-Orlicz space, and  bounded Hankel operators between some  Bergman-Orlicz spaces $A_\alpha^{\Phi_1}(\mathbb B^n)$ and $A_\alpha^{\Phi_2}(\mathbb B^n)$ where $\Phi_1$ and $\Phi_2$ are either convex or concave growth functions.
\end{abstract}
\maketitle

\section{Introduction}

Let  $\mathbb B^n$ be the unit ball  of $\mathbb C^n$. We denote by $d\nu$ the Lebesgue measure on $\mathbb B^n$ and $d\sigma$  the normalized measure on $\mathbb S^n=\partial{\mathbb B^n}$ (the boundary of
$\mathbb B^n$). The space $\mathcal H(\mathbb B^n)$ is the
set of holomorphic functions on $\mathbb B^n.$

For $z=(z_1,\cdots,z_n)$ and $w=(w_1,\cdots,w_n)$ in
$\C^n$, we let
$$\langle z,w\rangle =z_1\overline {w_1} +
\cdots + z_n\overline {w_n}$$
so that $|z|^2=\langle
z,z\rangle =|z_1|^2 +\cdots +|z_n|^2$.

  We say that a function $\Phi$ is a growth function if it is a continuous and non-decreasing function  from $[0,\infty)$ onto itself.

For $\alpha>-1$, we denote by $d\nu_{\alpha}$ the normalized Lebesgue measure $d\nu_{\alpha}(z)=c_{\alpha}(1-|z|^2)^{\alpha}d\nu(z)$, with $c_\alpha$ such
that $\nu_\alpha(\mathbb B^n)=1$. For $\Phi$ a growth function, the weighted Bergman-Orlicz space $\mathcal A_\alpha^{\Phi}(\mathbb B^n)$ is the space of holomorphic functions $f$ such that
$$||f||_{\alpha,\Phi}:=\int_{\mathbb B^n}\Phi(|f(z)|)d\nu_{\alpha}(z)<\infty.$$
More generally, we will also consider the Orlicz space $L_\alpha^{\Phi}(\mathbb B^n)$, that is, the space of functions $f$ such that
$$||f||_{\alpha,\Phi}:=\int_{\mathbb B^n}\Phi(|f(z)|)d\nu_{\alpha}(z)<\infty.$$
Hence $\mathcal A_\alpha^{\Phi}(\mathbb B^n)$ is the subspace
of $L_\alpha^{\Phi}(\mathbb B^n)$ consisting of holomorphic functions.

We define on $\mathcal A_\alpha^{\Phi}(\mathbb B^n)$ the following (quasi)-norm
\begin{equation}\label{BergOrdef1}
||f||^{lux}_{\alpha,\Phi}:=\inf\{\lambda>0: \int_{\mathbb B^n}\Phi\left(\frac{|f(z)|}{\lambda}\right)d\nu_{\alpha}(z)\le 1\}
\end{equation}
which is finite for $f\in \mathcal A_\alpha^{\Phi}(\mathbb B^n)$ (see \cite{sehbastevic}).

When  $\Phi(t)=t^p$, we recover the classical weighted Bergman spaces denoted by $\mathcal A_\alpha^{p}(\mathbb B^n)$   and defined by
$$\|f\|_{p,\alpha}^p=||f||_{\mathcal A_\alpha^{p}}^p:= \int_{\mathbb B^n}|f(z)|^pd\nu_{\alpha}(z)<\infty.$$
We say that a growth function $\Phi$ is of upper type  $q \geq 1$ if there exists $C>0$ such that, for $s>0$ and $t\ge 1$,
\begin{equation}\label{uppertype}
 \Phi(st)\le Ct^q\Phi(s).\end{equation}
We denote by $\mathscr{U}^q$ the set of growth functions $\Phi$ of upper type $q$, (for some $q\ge 1$), such that the function $t\mapsto \frac{\Phi(t)}{t}$ is non-decreasing.

We say that $\Phi$ is of lower type $p > 0$ if there exists $C>0$ such that, for $s>0$ and $0<t\le 1$,
\begin{equation}\label{eq:lowertype}
 \Phi(st)\le Ct^p\Phi(s).\end{equation}
We denote by $\mathscr{L}_p$ the set of growth functions $\Phi$ of lower type $p$,  (for some $p\le 1$), such that the function $t\mapsto \frac{\Phi(t)}{t}$ is non-increasing.

We say that $\Phi$ satisfies the $\Delta_2$-condition if there exists a constant $K>1$ such that, for any $t\ge 0$,
\begin{equation}\label{eq:delta2condition}
 \Phi(2t)\le K\Phi(t).\end{equation}

Recall that two growth functions $\Phi_1$ and $\Phi_2$ are said equivalent if there exists some constant $c$ such
that
$$c\Phi_1(ct) \le \Phi_2(t)\le c^{-1}\Phi_1(c^{-1}t).$$
Such equivalent growth functions define the same Orlicz space.
Note that we may always suppose that any $\Phi\in \mathscr{L}_p$ (resp. $\mathscr{U}_q$),  is concave (resp. convex) and
that $\Phi$ is a $\mathscr{C}^1$ function with derivative $\Phi^{\prime}(t)\backsimeq \frac{\Phi(t)}{t}$ (see \cite{BS} for the lower type functions).
\vskip .2cm

\vskip .2cm
Let us observe that if $\Phi$ is of upper type (resp. lower type) $p_1$, then it is of upper type (resp. lower type) $p_2$ for any $\infty >p_2>p_1$ (resp. $p_2<p_1<\infty$). Hence, when we say $\Phi\in \mathscr{U}^q$ (resp. $\Phi\in \mathscr{L}_p$), we suppose that $q$ (resp. $p$) is the smallest (resp. biggest) number $q_1$ (resp. $p_1$) such that $\Phi$ is of upper type $q_1$ (resp. lower type $p_1$).

\vskip .4cm
 We denoted by $L_\alpha^p( \mathbb
 B^n)$, $0<p<\infty$, the Lebesgue space with respect to the measure
 $d\nu_\alpha$.

The orthogonal projection of $L_\alpha^2( \mathbb
 B^n)$ onto $\mathcal {A}_\alpha^2(\mathbb B^n)$ is called the
 Bergman projection and denoted $P_\alpha$. It is given by
 \begin{equation} P_\alpha(f)(z)=\int_{\mathbb B^n}K_\alpha(z,\xi)f(\xi)d\nu_\alpha(\xi),\end{equation}
 where
$$K_\alpha(z,\xi)=\frac{1}{(1-\langle z,\xi \rangle)^{n+1+\alpha}}$$
 is  the weighted Bergman kernel on $ \mathbb B^n$. We denote as
 well by $P_\alpha$ its extension to $L_\alpha^1( \mathbb
 B^n)$.

\vskip .2cm
It is well known that the Bergman projection $P_\alpha$ is bounded on $L_\alpha^p( \mathbb
 B^n)$ for all $p\in (1,\infty)$. One of the consequences of this result is the fact that the topological dual
 space of the Bergman space $A_\alpha^p( \mathbb
 B^n)$ identifies with $A_\alpha^q( \mathbb
 B^n)$, with $\frac{1}{p}+\frac{1}{q}=1$, under the integral pairing
 \begin{equation}\label{eq:dualityA2}
 \langle f,g\rangle_\alpha:=\int_{\mathbb B^n}f(z)\overline {g(z)}d\nu_\alpha(z),
 \end{equation}
$f\in L_\alpha^p( \mathbb B^n)$, $g\in L_\alpha^q( \mathbb  B^n)$.
\vskip .2cm
The boundedness of the Bergman projection has been extended to the setting of Orlicz spaces for the class of
Young functions in \cite{DHZZ}, and this provides as a consequence that the dual space of the Bergman-Orlicz space
$A_\alpha^{\Phi}(\mathbb B^n)$ can be identifies with another Bergman-Orlicz space that we will specify in the next section.
\vskip .2cm
Our first interest in this paper is the characterization of the dual space of the Bergman-Orlicz spaces defined from the
class $\mathscr{L}_p$. We note that this class generalizes the class of power functions $\phi(t)=t^p$, $0<p<1$, and we have
that for $\Phi\in \mathscr{L}_p$,  the following inclusions hold
\begin{equation}\label{eq:doubleinclus}
\mathcal A_{\alpha}^{1}(\mathbb B^n)\subset \mathcal A_{\alpha}^{\Phi}(\mathbb B^n)\subset \mathcal A_{\alpha}^{p}(\mathbb B^n).
\end{equation}
We recall
 that given an analytic function $f$ on $\mathbb B^n$, the
 radial derivative $Rf$ of $f$ is defined by
 $$Rf(z)=\sum_{j=1}^{n}z_j \frac{\partial f}{\partial z_j}(z).$$

For $\beta\ge 0$, we denote by $\Gamma_\beta(\mathbb B^n)$
the space of holomorphic functions $f$ for which there
exists and integer $k>\beta$ and a positive constant $C$
such that
 $$ |R^kf(z)|\le C(1-|z|^2)^{\beta-k}.$$
Remark that for $\beta=0$, the class $\Gamma_\beta(\mathbb B^n)$
coincides with the usual Bloch class $\mathcal B$. The Bloch class is the space
of holomorphic functions in $\mathbb B^n$ such that
$$\sup_{z\in\mathbb B^n}|Rf(z)|(1-|z|^2) <\infty.$$
 For $\beta>0$, it
coincides with the class of Lipschitz functions of order
$\beta$.

It is known that, for $0<p\le 1$, the dual of the Bergman space $\mathcal A_{\alpha}^{p}(\mathbb B^n)$,  coincides with $\Gamma_\beta(\mathbb B^n)$ with $\beta=(n+1+\alpha)(\frac{1}{p}-1)$ under the integral pairing

 \begin{equation}\label{eq:dualitylarge}
 \lim_{r\rightarrow 1}\int_{\mathbb B^n}f(rz)\overline {g(z)}d\nu_\alpha(z)
 \end{equation}
(see \cite{KZ}).
\vskip .2cm
To $\Phi$ a growth, we associate the function $$\rho(t)=\frac{1}{t\Phi^{-1}(1/t)}.$$
The function $\rho$ is quite relevant in the study of Orlicz space of analytic functions (see \cite{BS, Janson, ST, ST1} and the references therein).  Note in particular that in the case of $A_\alpha^p(\mathbb B^n)$, $\Phi(t)=t^p$ and $\rho(t)=t^{\frac{1}{p}-1}$, hence  $f\in \Gamma_\beta(\mathbb B^n)$ can be written as $$ |R^kf(z)|\le C(1-|z|^2)^{-k}\rho\left((1-|z|^2)^{n+1+\alpha}\right).$$
From this observation, we will make the following generalization. Let $\rho$ be a positive continuous increasing function from $[0,\infty)$ onto itself. Let $\gamma>0$. We say that $\rho$ is of upper type $\gamma$ on $[0,1]$ if there exists a constant $C$ such that
\begin{eqnarray}\label{rho-uppertype}
    \rho(st)\le Cs^\gamma\rho(t),
\end{eqnarray}
for $s>1$ and $st\leq 1$. We will call a weight, a function $\rho$ which is a continuous increasing function from $[0,\infty)$ onto itself, which is of upper type $\gamma$, for some $\gamma>0$.

Now for $\alpha>-1$ and a weight $\rho$ (of upper type $\gamma$), we define  the weighted Lipschitz space $\Gamma_{\alpha,\rho}(\mathbb B^n)$ as the space of holomorphic functions $f$ in $\mathbb B^n$ such that, for some integer $k>\gamma (n+1+\alpha) $ and a positive constant $C>0$, we have
$$ |R^kf(z)|\le C(1-|z|^2)^{-k}\rho\left((1-|z|^2)^{n+1+\alpha}\right).$$
We will show that, as in the classical Lipschitz spaces,  these spaces are independent of $k$. This allows us to see $\Gamma_{\alpha,\rho}(\mathbb B^n)$ as a Banach space under the following norm
$$||f||_{\Gamma_{\alpha,\rho}(\mathbb B^n)} = |f(0)|+ \sup_{z\in\mathbb B^n}\frac{|R^k f(z)|(1-|z|^2)^{k}}{\rho\left((1-|z|^2)^{n+1+\alpha}\right)}.$$
The following is our first main result which extends the duality result for classical Bergman spaces with
small exponent to Bergman-Orlicz spaces with concave exponent.
\btheo\label{lem:dualitysmall}
Let $\alpha>-1$, $\Phi\in \mathscr{L}_p$ and $\rho(t)=\frac{1}{t\Phi^{-1}(1/t)}$. Then the topological dual space $\left(\mathcal A_{\alpha}^{\Phi}(\mathbb B^n)\right)^*$ of $\mathcal A_{\alpha}^{\Phi}(\mathbb B^n)$ identifies with $\Gamma_{\alpha, \rho}(\mathbb B^n)$ under the duality pairing
\begin{equation}\label{eq:dualpairing}
\langle f,g\rangle_\alpha:=\lim_{r\rightarrow 1}\int_{\mathbb B^n}f(rz)\overline {g(z)}d\nu_\alpha(z),
\end{equation}

where $f\in \mathcal A_{\alpha}^{\Phi}(\mathbb B^n)$ and $g\in \Gamma_{\alpha, \rho}(\mathbb B^n).$
\etheo

The proof of the above result required two main steps. First one has to insure that the above definition of the space  $\Gamma_{\alpha, \rho}(\mathbb B^n)$
does not depend on the choice of the number of derivatives. Second, one needs a nice example of functions in the Bergman-Orlicz space $\mathcal A_{\alpha}^{\Phi}(\mathbb B^n)$ and a generalization to large Bergman-Orlicz spaces of the following inequality known for Bergman spaces with exponent $0<p\le 1$ (see \cite{ZZ}),
\begin{equation}\label{eq:intestimsmall}
\int_{\mathbb B^n}|f(z)|(1-|z|^2)^{(\frac{1}{p}-1)(n+1+\alpha)}d\nu_\alpha(z)\le C\|f\|_{\alpha,p}^p.
\end{equation}
 \vskip .4cm

 For $b\in \mathcal {A}_\alpha^2(\mathbb B^n)$, the small Hankel
 operator with symbol $b$ is defined for $f$ a bounded
 holomorphic function by $h_b(f):=P_\alpha(b\overline f)$.

\vskip .2cm
 Boundedness of the small Hankel operator between classical weighted Bergman spaces has been considered in \cite{BL} where using duality and test functions, the authors obtained a full characterization of bounded Hankel operators between Bergman spaces except for estimations with loss i.e. $h_b:\mathcal {A}_\alpha^p(\mathbb B^n)\mapsto \mathcal {A}_\alpha^q(\mathbb B^n)$ with $1\le q< p<\infty$. The estimations with loss have been recently handled by J. Pau and R. Zhao in \cite{PauZhao2}, closing the subject for the classical weighted Bergman spaces.
 \vskip .2cm
 Our second interest  in this paper is for the boundedness of the small Hankel operators $h_b$ from $\mathcal {A}_\alpha^{\Phi_1}(\mathbb
 B^n)$  to $\mathcal {A}_\alpha^{\Phi_2}(\mathbb B^n)$. We do not use a specific method but combine several techniques some of them appearing in \cite{BL} or used in the case of Hardy-Orlicz spaces in \cite{ST, ST1}. In particular, when considering boundedness of $h_b$ on $\mathcal {A}_\alpha^\Phi (\mathbb B^n)$ with $\Phi\in \mathscr{U}^q$, we use a weak factorization result of the Bergman space $\mathcal {A}_\alpha^1(\mathbb B^n)$ in term of Bergman-Orlicz functions, extending the usual weak factorization for this space. Nevertheless, we do not generalize this method for the whole situation, as when $\Phi_1$ and $\Phi_2$ are growth functions with $\Phi_2\in \mathscr{U}^q$, we are dealing only with the upper triangle case, i.e  $\Phi_1^{-1}(t)\Psi_2^{-1}(t)\in \mathscr{U}^q$, $\Psi_2$ being the complementary function of $\Phi_2$ to be defined later.
 \vskip .2cm
The ranges of the symbols of bounded Hankel operators obtained here are some weighted Lipschitz spaces related to the dual spaces of
Bergman Orlicz spaces with concave growth functions as given above. This will allow us to study the boundedness of the Hankel operators between
Bergman-Orlicz spaces in the same range of growth functions
 as in \cite{BS,ST}.   
However, when $\Phi_2$ is a concave growth function, we will suppose that $\Phi_2$ satisfies  a Dini condition to be defined later.  This will cause additional restrictions on the growth functions $\Phi_1$ and $\Phi_2$ for which we are able to extend the Hankel operators, $h_b$, into bounded operators from
 $\mathcal A^{\Phi_1}_\alpha(\mathbb B^n)$ into $\mathcal A^{\Phi_2}_\alpha(\mathbb B^n)$.
In particular, we are able here to handle  the cases where the growth functions
are given by $\Phi_1(t)=\left(\frac t{\log(e+t)}\right)^p,\; p\leq 1$ and $\Phi_2(t)=\left(\frac t{\log(e+t)}\right)^s,\; s< 1$. Moreover, our results generalize the results obtained in \cite{BL}.
\vskip .2cm
The paper is organized as follows, in section \ref{section2}, we collect some results that are needed to characterize the dual spaces of large Bergman-Orlicz space and to study the boundedness propoerties of Hankel operators between Bergman-Orlicz spaces. In section \ref{section3}, we deal with the duality question for Bergman-Orlicz spaces with concave exponent. In the last section, each subsection is devoted to the study, in each case, of the boundedness of the Hankel operator $h_b$ from $\mathcal A^{\Phi_1}(\mathbb
 B^n)$ into $\mathcal A^{\Phi_2}(\mathbb
 B^n)$, when $\Phi_i\in \mathscr{L}_p \;\textrm{or}\;\mathscr{U}^q,\;i=1,2.$
\vskip .2cm
 Finally, all over the text, $C$ will be a constant not necessary the same at each occurrence. We will also use the notation $C(k)$
 to express the fact that the constant depends on the underlined parameter. Given two positive quantities $A$ and $B$, the notation
 $A\lesssim B$ means that $A\le CB$ for some positive constant $C$. When $A\lesssim B$ and $B\lesssim A$, we write $A\backsimeq B$.

\section{Preliminaries}\label{section2}
In this section, we recall some known results that are needed in our study, we also extend to Orlicz setting many some classical results known for the Bergman spaces.
\subsection{Some properties of growth functions}
We collect in this subsection some properties of growth functions we shall used later. For $\Phi$ a convex growth function, we recall that the complementary function, $\Psi : \mathbb R_+ \rightarrow \mathbb R_+$, is defined by
\begin{equation}\label{complementarydefinition}
\Psi(s)=\sup_{t\in\mathbb R_+}\{ts - \Phi(t)\}.
\end{equation}
 One easily checks that if $\Phi\in \mathscr{U}^q$, then $\Psi$ is also a growth function of lower type such that $t\mapsto \frac{\Psi(t)}{t}$ is non-decreasing but which may not  satisfy the $\Delta_2-$conditon. We say that the growth function $\Phi$ satisfies the $\bigtriangledown_2-$condition whenever both $\Phi$ and its complementary satisfy the $\Delta_2-$conditon.
\vskip .2cm
For $\Phi$ a $\mathcal C^1$ growth function, the lower and the upper indices of $\Phi$ are respectively defined by
$$a_\Phi:=\inf_{t>0}\frac{t\Phi^\prime(t)}{\Phi(t)}\,\,\,\textrm{and}\,\,\,b_\Phi:=\sup_{t>0}\frac{t\Phi^\prime(t)}{\Phi(t)}.$$
We recall that when $\Phi$ is convex, then $1\le a_\Phi\le b_\Phi<\infty$ and, if $\Phi$ is concave, then $0<a_\Phi\le b_\Phi\le 1$. We have the following useful fact.
\blem\label{indices}
Let $\Phi$ be a $\mathcal C^1$ growth function. Denote by $p$ and $q$ its lower and its upper indices respectively. Then the functions $\frac{\Phi(t)}{t^p}$ and $\frac{\Phi^{-1}(t)}{t^{1/q}}$ are increasing.
\elem
\begin{proof}
We only prove that $\frac{\Phi(t)}{t^p}$ is increasing. The proof is the same for the second function. Recall that by definition of $p$, we have $p\leq\frac{t\Phi^\prime(t)}{\Phi(t)}$ for any $t>0$. It easily follows that
$$\left(\frac{\Phi(t)}{t^p}\right)^\prime=\frac{\Phi^\prime(t)}{t^p}-p\frac{\Phi(t)}{t^{p+1}}\geq \frac{\Phi^\prime(t)}{t^p}-\frac{t\Phi^\prime(t)}{\Phi(t)}\times\frac{\Phi(t)}{t^{p+1}}=0.$$
The proof is complete.
\end{proof}
\begin{remark}\label{about-indices}
One useful way to use this lemma is to observe that it allow us to say that, if $\Phi\in \mathscr{L}_p$, then the growth function $\Phi_p$, defined by
$\Phi_p(t)=\Phi(t^{1/p})$, is in $\mathscr{U}^q$ for some $q\geq 1$. So we may assume that $\Phi_p$ is convex.


We also observe that $a_\Phi$ (resp. $b_\Phi$) coincides with the biggest (resp. smallest) number $p$ such that $\Phi$ is of lower (resp. upper) type $p$.
\end{remark}
\vskip .2cm
 We say that  $\Phi\in \mathscr{U}^q$ satisfies the Dini condition if there exists a  constant $C>0$ such that, for $t>0$,
\begin{equation}\label{dinicondition}
\int_0^t\frac{\Phi(s)}{s^2}ds\leq C\frac{\Phi(t)}{t}.
\end{equation}
We observe that if $\Phi$ satisfies (\ref{dinicondition}), then $\Phi$ satisfies the $\bigtriangledown_2-$condition.

\vskip .2cm
We will also make use of the following  properties of growth functions
established in section $2$ of \cite{ST}. We recall them here for quick references.
\bprop\label{phiandinverse}
The following assertion holds:
\begin{center}
    $\Phi\in \mathscr{L}_p$ if and only if $\Phi^{-1}\in \mathscr{U}^{1/p}.$
\end{center}
\eprop
\blem\label{stabilityoflowertypeclass}
Let $\Phi_1\in \mathscr{L}_p$ and $\Phi_2\in \mathscr{U}^q$, and $\Psi_2$ the complementary function of $\Phi_2.$ Let $\Phi$ be such that
\begin{equation*}
\Phi^{-1}(t):=\Phi_1^{-1}(t)\Psi_2^{-1}(t).
\end{equation*}
Then $\Phi\in \mathscr{L}_r$ for some $r\leq p.$
\elem
\blem\label{conditionforlowertype1}
Let $\Phi_1$ be a growth function and $\Phi_2\in \mathscr{U}^q$, $\rho_i(t)=\frac{1}{t\Phi_i^{-1}(1/t)}$ and $\Psi_2$ the complementary of $\Phi_2.$ Then, if
$$\rho_{\Phi}:=\frac{\rho_1}{\rho_2},$$
we also have
\begin{equation}\label{equivcompland}
\Phi^{-1}(t)\simeq \Phi_1^{-1}(t)\Psi_2^{-1}(t)
\end{equation}
and vice-versa.
\elem
\blem\label{conditionforlowertype}
Let $\Phi_1$ and $\Phi_2$ be in $ \mathscr{U}^q$, and $\Psi_2$ the complementary function of $\Phi_2$.  Let $\Phi$ be such that $\Phi^{-1}(t)=\Phi_1^{-1}(t)\Psi_2^{-1}(t)$. We suppose that $\Phi_2$ satisfies the Dini condition (\ref{dinicondition}) and that $$\frac{\Phi_2^{-1}\circ\Phi_1(t)}{t}\quad\textrm{is non-increasing}.$$
Then $\Phi\in \mathscr{L}_p$ for some $p>0$.
\elem

The following can be adapted from \cite{VT}.
\begin{proposition}\label{prop:volbergtolokonnikov}
For $\Phi_1$ and $\Phi_2$ two growth functions of  lower type, $\alpha>-1$, the bilinear map $(f,g)\mapsto fg$ sends $L_\alpha^{\Phi_1}\times L_\alpha^{\Phi_2}$ onto $L_\alpha^{\Phi}$, with the inverse mappings of $\Phi_1,\Phi_2$ and $\Phi$ related by
\begin{equation}\label{eq:volbergproduct}
\Phi^{-1}=\Phi_1^{-1}\times \Phi_2^{-1}.
\end{equation}
Moreover, there exists some constant $c$ such that
$$||fg||^{lux}_{L_\alpha^{\Phi}}\le c ||f||^{lux}_{L_\alpha^{\Phi_1}}||g||^{lux}_{L_\alpha^{\Phi_2}}.$$
\end{proposition}

\subsection{Boundedness of the Bergman projection}
  We start by recalling  the following result in \cite{Bekolle}.
 \begin{proposition}\label{prop:weakboundedness}
 Let $\alpha>-1$,  there exists a constant $C>0$ such that for $f\in \mathcal {A}_\alpha^1(\mathbb B^n)$
 $$\nu_\alpha\left(\{z\in \mathbb B^n: |P_\alpha f(z)|>\lambda\}\right)\le C\frac{\|f\|_{1,\alpha}}{\lambda}.$$
 \end{proposition}

The next result follows from interpolation with Orlicz functions (see \cite{DHZZ}).
\begin{proposition}
Let $\alpha>-1$ and $\Phi\in \mathscr{U}^q$. Suppose that $\Phi$ satisfies the $\bigtriangledown_2-$condition. Then the Bergman projection $P_\alpha$ extends into a bounded operator on $L_\alpha^\Phi (\mathbb B^n)$.
\end{proposition}
From the duality result in \cite{RR}, since $P_\alpha$ is bounded in
$L^\Phi_\alpha$ for $\Phi\in \mathscr{U}^q$ that satisfies the
$\bigtriangledown_2-$condition, we obtain the following duality result in this case.
\begin{proposition}
Let $\Phi\in \mathscr{U}^q$ and $\alpha>-1$. Suppose that  $\Phi$
satisfies the $\bigtriangledown_2-$condition and denote by $\Psi$ its complementary function. Then the dual space $\left(\mathcal A_{\alpha}^{\Phi}(\mathbb B^n)\right)^*$ of $\mathcal A_{\alpha}^{\Phi}(\mathbb B^n)$ identifies with $\mathcal A_{\alpha}^{\Psi}(\mathbb B^n)$
under the integral pairing
$$\langle f,g\rangle_\alpha=\int_{\mathbb B^n}f(z)\overline {g(z)}d\nu_\alpha (z),\,\,\,f\in \mathcal A_{\alpha}^{\Phi}(\mathbb B^n),\,\,\,g\in \mathcal A_{\alpha}^{\Psi}(\mathbb B^n).$$
\end{proposition}
\subsection{Lipschitz-type spaces}
We consider in this subsection some weighted Lipschitz spaces and their logarithmic counterparts.
We recall
 that given an analytic function $f$ on $\mathbb B^n$, the
 radial derivative $Rf$ of $f$ is defined by
 $$Rf(z)=\sum_{j=1}^{n}z_j \frac{\partial f}{\partial z_j}(z).$$

For $\alpha>-1$ and a weight $\rho$ (of upper type $\gamma$), we recall that the weighted Lipschitz space $\Gamma_{\alpha,\rho}(\mathbb B^n)$ has been defined as the space of holomorphic functions $f$ in $\mathbb B^n$ such that, for some integer $k>\gamma (n+1+\alpha) $ and a positive constant $C>0$, we have
$$ |R^kf(z)|\le C(1-|z|^2)^{-k}\rho\left((1-|z|^2)^{n+1+\alpha}\right).$$
We will also need a logarithmic version of the above space, $L\Gamma_{\alpha,\rho}(\mathbb B^n)$, defined as the space of holomorphic functions $f$ in $\mathbb B^n$ such that, for some $k>\gamma (n+1+\alpha)$ and a positive constant $C>0$, we have
$$ |R^kf(z)|\le C(1-|z|^2)^{-k}\rho\left((1-|z|^2)^{n+1+\alpha}\right)\left(\log\frac{1}{1-|z|^2}\right)^{-1}.$$
One can show that, as in the classical Lipschitz spaces,  these spaces are independent of $k$. We prove this in the following proposition.
\begin{proposition}\label{independance-ofk}
Let $\rho$ be a weight satisfying (\ref{rho-uppertype}).
The weighted (resp. logarithmic weighted) Lipschitz spaces   $\Gamma_{\alpha,\rho}(\mathbb B^n)$ (resp. $L\Gamma_{\alpha,\rho}(\mathbb B^n)$) are independent of various values of $k$.
\end{proposition}
\begin{proof}

Let us provide a proof for $\Gamma_{\alpha,\rho}(\mathbb B^n)$, the proof for $L\Gamma_{\alpha,\rho}(\mathbb B^n)$ requires only few harmless modifications.

Let $f$ be an holomorphic function in $\mathbb B^n$. Let us first suppose that there is a constant $C>0$ such that for some nonnegative integer $k>(n+1+\alpha)\gamma$ with $\gamma$ as in (\ref{rho-uppertype}), and any $z\in \mathbb B^n$,
 \Be\label{eq:lipfork}
 |R^kf(z)|\le C(1-|z|^2)^{-k}\rho\left((1-|z|^2)^{n+1+\alpha}\right).
 \Ee
We want to show that (\ref{eq:lipfork})  holds for $k+1$.

 From (\ref{eq:lipfork}), it follows in particular that as $\rho$ is increasing on $(0,1]$, the function $|R^kf(z)|(1-|z|^2)^{k}$ is bounded on $\mathbb B^n$. Hence, for $\beta$ sufficiently large (in fact $\beta>k-\alpha-1$ will do ), we have the representation
 $$R^kf(z)=C_{\alpha\beta}\int_{\mathbb B^n}\frac{R^kf(w)(1-|w|^2)^{\beta}}{\left(1-\langle z,w\rangle\right)^{n+1+\beta+\alpha}}d\nu_\alpha(w).$$
 Thus
 \Beas
 |R^{k+1}f(z)| &=& \left|C_{\alpha\beta}\int_{\mathbb B^n}\frac{(n+1+\beta+\alpha)\langle z,w\rangle R^kf(w)(1-|w|^2)^{\beta}}{\left(1-\langle z,w\rangle\right)^{n+2+\beta+\alpha}}d\nu_\alpha(w)\right|\\ &\lesssim& \int_{\mathbb B^n}\frac{|R^kf(w)|(1-|w|^2)^{\beta}}{|1-\langle z,w\rangle|^{n+2+\beta+\alpha}}d\nu_\alpha(w)\\ &\lesssim& \int_{\mathbb B^n}\frac{\rho\left((1-|w|^2)^{n+1+\alpha}\right)(1-|w|^2)^{\beta-k}}{|1-\langle z,w\rangle|^{n+2+\beta+\alpha}}d\nu_\alpha(w)\\ &=& I_1+I_2,
 \Eeas
where
$$I_1=\int_{1-|w|^2\le 1-|z|^2}\frac{\rho\left((1-|w|^2)^{n+1+\alpha}\right)(1-|w|^2)^{\beta-k}}{|1-\langle z,w\rangle|^{n+2+\beta+\alpha}}d\nu_\alpha(w)$$
and
$$I_2=\int_{1-|w|^2> 1-|z|^2}\frac{\rho\left((1-|w|^2)^{n+1+\alpha}\right)(1-|w|^2)^{\beta-k}}{|1-\langle z,w\rangle|^{n+2+\beta+\alpha}}d\nu_\alpha(w).$$
Let us start by estimating the integral $I_1$. Using the monotonicity of the weight $\rho$ and \cite[Propostion 1.4.10]{R}, we obtain
 \Beas
 I_1 &=& \int_{1-|w|^2\le 1-|z|^2}\frac{\rho\left((1-|w|^2)^{n+1+\alpha}\right)(1-|w|^2)^{\beta-k}}{|1-\langle z,w\rangle|^{n+2+\beta+\alpha}}d\nu_\alpha(w)\\
&\lesssim& \rho\left((1-|z|^2)^{n+1+\alpha}\right)\int_{1-|w|^2\le 1-|z|^2}\frac{(1-|w|^2)^{\beta-k}}{|1-\langle z,w\rangle|^{n+2+\beta+\alpha}}d\nu_\alpha(w)\\
 &\lesssim& (1-|z|^2)^{-k-1}\rho\left((1-|z|^2)^{n+1+\alpha}\right).
 \Eeas
 To handle the integral $I_2$, we use (\ref{rho-uppertype}) and \cite[Proposition 1.4.10]{R} once more, and we obtain
 \Beas
 I_2 &=& \int_{1-|w|^2> 1-|z|^2}\frac{\rho\left((1-|w|^2)^{n+1+\alpha}\right)(1-|w|^2)^{\beta-k}}{|1-\langle z,w\rangle|^{n+2+\beta+\alpha}}d\nu_\alpha(w)\\ &\le& C\frac{\rho\left((1-|z|^2)^{n+1+\alpha}\right)}{(1-|z|^2)^{(n+1+\alpha)\gamma}}\int_{\mathbb B^n}\frac{(1-|w|^2)^{\beta-k+(n+1+\alpha)\gamma}}{|1-\langle z,w\rangle|^{n+2+\beta+\alpha}}d\nu_\alpha(w)\\ &\le& C\frac{\rho\left((1-|z|^2)^{n+1+\alpha}\right)}{(1-|z|^2)^{(n+1+\alpha)\gamma}}\times (1-|z|^2)^{-k-1+(n+1+\alpha)\gamma}\\ &=& C(1-|z|^2)^{-k-1}\rho\left((1-|z|^2)^{n+1+\alpha}\right).
 \Eeas
 We conclude that there is also a constant $C>0$ such that for any $z\in \mathbb B^n$,
 $$|R^{k+1}f(z)|\le C(1-|z|^2)^{-k-1}\rho\left((1-|z|^2)^{n+1+\alpha}\right).$$

Now, let us suppose that there is a constant $C>0$ such that for some integer $k>(n+1+\alpha)\gamma$ with $\gamma$ being the upper-type of $\rho$, and any $z\in \mathbb B^n$,
\Bea\label{converse:independentofk}
|R^{k+1}f(z)|\le C(1-|z|^2)^{-k-1}\rho\left((1-|z|^2)^{n+1+\alpha}\right).
\Eea
We will show that this implies that we can find a constant $\tilde {C}>0$ such that for any $z\in \mathbb B^n$,
$$|R^{k}f(z)|\le \tilde {C}(1-|z|^2)^{-k}\rho\left((1-|z|^2)^{n+1+\alpha}\right).$$
We start by the following lemma.
\blem
Let $\rho$ be a weight. Then for any integer $k>(n+1+\alpha)\gamma$, there is a constant $\tilde {C}$ depending on $n,\alpha,C,\gamma$ where $C$ and $\gamma$ are as in (\ref{rho-uppertype}), such that for any $0<t<1$,
\Be\label{eq:dinitype}
\int_t^1\frac{\rho(s^{n+1+\alpha})}{s^{k+1}}ds\le \tilde {C}\frac{\rho(t^{n+1+\alpha})}{t^{k}}.
\Ee
\elem
\begin{proof}
Let $N$ be the smallest integer such that $t2^{N+1}\ge 1$. Then we have using the monotonicity of the function $\rho$ and (\ref{rho-uppertype}) that
\Beas
\int_t^1\frac{\rho(s^{n+1+\alpha})}{s^{k+1}}ds &\le& \sum_{l=0}^N\int_{t2^l}^{t2^{l+1}}\frac{\rho(s^{n+1+\alpha})}{s^{k+1}}ds\\ &\le& \sum_{l=0}^N\frac{\rho(t^{n+1+\alpha}2^{(l+1)(n+1+\alpha)})}{t^{k+1}2^{l(k+1)}}2^lt\\ &=& \sum_{l=0}^N\frac{\rho(t^{n+1+\alpha}2^{(l+1)(n+1+\alpha)})}{t^{k}2^{lk}}\\ &\le& C2^{(n+1+\alpha)\gamma}\sum_{l=0}^N\frac{\rho(t^{n+1+\alpha})}{t^{k}}2^{-l(k-(n+1+\alpha)\gamma)}\\ &=& C2^{(n+1+\alpha)\gamma}\frac{\rho(t^{n+1+\alpha})}{t^{k}}\sum_{l=0}^N2^{-l(k-(n+1+\alpha)\gamma)}\\ &\le& \tilde {C}\frac{\rho(t^{n+1+\alpha})}{t^{k}}.
\Eeas
The proof of the lemma is complete.
\end{proof}
Coming back to our proof, we first consider the case of $|z|>1/2$, $z=r\xi$, $\xi\in \mathbb S^n$. We recall the following identity for any holomorphic function $g$:
\Be\label{eq:gradientidentity}
g(z)-g(r\xi /2)=\int_{1/2}^1\frac{Rg(tz)}{t}dt.
\Ee
Applying (\ref{eq:gradientidentity}) with $g=R^kf$, we obtain using (\ref{eq:dinitype}) that for $|z|>1/2$,
\Beas
|R^kf(z)-R^kf(r\xi /2)| &\le& \int_{1/2}^1\frac{\left|R^{k+1}g(tz)\right|}{t}dt\\ &\le& C\int_{1/2}^1(1-t|z|)^{-k-1}\rho\left((1-t|z|)^{n+1+\alpha}\right)dt\\
 &\le& C\int_{1-|z|}^1\frac{\rho(s^{n+1+\alpha})}{s^{k+1}}ds\\
 &\le& C_1(1-|z|^2)^{-k}\rho\left((1-|z|^2)^{n+1+\alpha}\right).
\Eeas
Hence for $|z|>1/2$,
$$|R^kf(z)|\le S(1/2)+C_1(1-|z|^2)^{-k}\rho\left((1-|z|^2)^{n+1+\alpha}\right)$$
with $$S(1/2)=\max_{|z|\le 1/2}|R^kf(z)|.$$
Next for $|z|\le 1/2$, applying the mean value property to $R^kf(z)-R^kf(0)$, and (\ref{converse:independentofk}), we obtain
\Bea\label{eq:estimate:small:z}
 \max_{|z|\le 1/2}|R^kf(z)|^2 &=& \max_{|z|\le 1/2}|R^kf(z)-R^kf(0)|^2 \nonumber \\
&\le& 4^n\int_{|w|\le 3/4}|R^kf(w)-R^kf(0)|^2d\nu(w) \nonumber \\
&\le& 4^n\int_{|w|\le 3/4}|R^{k+1}f(w)|^2d\nu(w) \nonumber \\
&\le&
3^n\max_{|z|\le 3/4}|R^{k+1}f(z)|^2 \nonumber \\
 &\le& 3^nC^2\max_{|z|\le 3/4}(1-|z|^2)^{-2k}\rho\left((1-|z|^2)^{n+1+\alpha}\right)^2 \nonumber \\
 &\le& 3^nC^2\rho(1)^2.
\Eea
Hence using the latter and the fact that if $C_0$ is the constant in (\ref{rho-uppertype}), then for $z\in \mathbb B^n$,
\begin{multline}\label{eq:inverserhoupper}
(1-|z|^2)^{-k}\rho\left((1-|z|^2)^{n+1+\alpha}\right)\\
\ge (1-|z|^2)^{-(n+1+\alpha)\gamma}\rho\left((1-|z|^2)^{n+1+\alpha}\right)\ge \frac{\rho(1)}{C_0},
\end{multline}
we obtain
\Be\label{eq:estimateofS}
S(1/2)=\max_{|z|\le 1/2}|R^kf(z)|\le \left((\sqrt {3})^nC\right)C_0(1-|z|^2)^{-k}\rho\left((1-|z|^2)^{n+1+\alpha}\right).
\Ee
Thus for $|z|>1/2$,
\Be\label{eq:estimateforbigmod}
|R^kf(z)|\le C_1(1-|z|^2)^{-k}\rho\left((1-|z|^2)^{n+1+\alpha}\right).
\Ee

For $|z|\le 1/2$, using (\ref{eq:estimate:small:z}) and (\ref{eq:inverserhoupper}) we have
that for $|z|\le 1/2$,

\Beas
|R^kf(z)| &\lesssim& C_2(1-|z|^2)^{-k}\rho\left((1-|z|^2)^{n+1+\alpha}\right).
\Eeas
Thus taking $\tilde C=\max\{C_1,C_2\}$, we obtain that for any $z\in \mathbb B^n$,
$$|R^kf(z)|\le \tilde {C}(1-|z|^2)^{-k}\rho\left((1-|z|^2)^{n+1+\alpha}\right).$$
The proof is complete.
\end{proof}

As a consequence, the spaces, $\Gamma_{\alpha,\rho}(\mathbb B^n)$, $L\Gamma_{\alpha,\rho}(\mathbb B^n)$ 
become Banach spaces under the following norms
\begin{eqnarray*}
||f||_{\Gamma_{\alpha,\rho}(\mathbb B^n)}&=& |f(0)|+ \sup_{z\in\mathbb B^n}\frac{|R^k f(z)|(1-|z|^2)^{k}}{\rho\left((1-|z|^2)^{n+1+\alpha}\right)}\\
||f||_{L\Gamma_{\alpha,\rho}(\mathbb B^n)}&=& |f(0)|+ \sup_{z\in\mathbb B^n}\frac{|R^k f(z)|(1-|z|^2)^{k}}{\rho\left((1-|z|^2)^{n+1+\alpha}\right)}|\log(1-|z|^2)|,
\end{eqnarray*}
where $k$ is a fixed integer strictly greater than $\gamma (n+1+\alpha)$.

We will show that the space  $\Gamma_{\alpha,\rho}(\mathbb B^n)$ is the topological dual space of $\mathcal A_{\alpha}^{\Phi}(\mathbb B^n)$  when $\rho$ and $\Phi$ are related by
 $$\rho(t):=\frac{1}{t\Phi^{-1}(1/t)}.$$
We recall the following fact from \cite[Proposition 3.10]{V}.
\begin{proposition} Let $\Phi\in \mathscr{L}_p$ and $\rho(t):=\frac{1}{t\Phi^{-1}(1/t)}$, then $\rho$ is a weight of upper type $\frac 1p -1$.
\end{proposition}

 We will need the following differential operator of order $k$:
$$M_k^{\alpha}=[(n+k+\alpha)I+R]\cdots [(n+1+\alpha)I+R]$$ where $I$ is the
identity operator, $k\in \mathbb N^*$. 
The following lemma
can be proved using integration by parts.

\blem\label{integrationparts} Let $f,g$ be holomorphic
polynomials on $\mathbb B^n$. Then the following equality
holds 
$$\int_{\mathbb B^n}f(z)\overline
g(z)d\nu_\alpha(z) = C_{k,\alpha}\int_{\mathbb
B^n}f(z)\overline
{M_{k}^\alpha g(z)}(1-|z|^2)^{k}d\nu_\alpha (z),$$
where $C_{k,\alpha}$ is a constant depending only on $k$ and $\alpha$.
\elem

The following lemma is Lemma 2.2 of
\cite{BL}.
\blem\label{AlineLuo} Let $(a_j)$ be a sequence of positive numbers,
and let $L_k$ be the differential operator of order $k$
defined by
$$L_k:=(a_0I+R)(a_1I+R)\cdots (a_{k-1}I+R).$$
Then $f$ belongs to $\Gamma_\beta(\mathbb B^n)$ if and
only if there exist and integer $k>\beta$ and a positive
constant $C$ such that $$|L_kf(z)|\le
C(1-|z|^2)^{\beta-k}.$$\elem

As remarked in \cite{BL}, the equivalence in the above
lemma also holds if we multiply the right hand side of the
last inequality by a logarithmic terms. That is, for fixed $t\in \mathbb R$,
$$|R^kf(z)|\le C(1-|z|^2)^{\beta-k}|\log(1-|z|^2)|^t$$
if and only if $$|L_kf(z)|\le
C(1-|z|^2)^{\beta-k}|\log(1-|z|^2)|^t.$$

Using the same kind of techniques as in Proposition \ref{independance-ofk},
we can also show that for a weight $\rho$ of upper type $\gamma$, $f$ belongs to $\Gamma_{\alpha,\rho}(\mathbb B^n)$ 
if and only if there exist an integer $k>\gamma(n+1+\alpha)$ and a positive
constant $C$ such that $$|L_kf(z)|\le
C(1-|z|^2)^{-k}\rho\left((1-|z|^2)^{n+1+\alpha}\right).$$
The same is true for $L\Gamma_{\alpha,\rho}(\mathbb B^n)$.

As a consequence of this fact, we have that
\begin{eqnarray}
||f||_{\Gamma_{\alpha,\rho}(\mathbb B^n)}&\backsimeq & |f(0)|+\sup_{z\in\mathbb B^n}\frac{|M_k^\alpha f(z)|(1-|z|^2)^{k}}{\rho\left((1-|z|^2)^{n+1+\alpha}\right)}\\
||f||_{L\Gamma_{\alpha,\rho}(\mathbb B^n)}&\backsimeq & |f(0)|+\sup_{z\in\mathbb B^n}\frac{|M_k^\alpha f(z)|(1-|z|^2)^{k}}{\rho\left((1-|z|^2)^{n+1+\alpha}\right)}|\log(1-|z|^2)|
\end{eqnarray}
where $k$ is a fixed integer strictly greater than $\gamma (n+1+\alpha)$.
\vskip .4cm
\subsection{Some useful estimates}
The next proposition gives pointwise estimates for functions in
$\mathcal A_{\alpha}^{\Phi}(\mathbb B^n)$, $\Phi\in \mathscr{L}_p$.
\begin{lemma}\label{pointwise-concave}
Let $\Phi\in \mathscr{L}_p$ and $\alpha>-1$. There is a constant $C>1$ such that
 for any $f\in \mathcal A_{\alpha}^{\Phi}(\mathbb B^n)$,
\begin{equation}\label{eq:pointwise-concave}
|f(z)|\le C\Phi^{-1}\left(\frac 1{(1-|z|^2)^{n+1+\alpha}}\right)\|f\|_{\alpha,\Phi}^{lux}.
\end{equation}
\end{lemma}
\begin{proof}
Let $f\in \mathcal A_{\alpha}^{\Phi}(\mathbb B^n)$. Note that if $\|f\|_{\alpha,\Phi}^{lux}=0$, then $f=0$ a.e and consequently, we obviously have (\ref{eq:pointwise-concave}). Let us assume that $\|f\|_{\alpha,\Phi}^{lux}\neq 0$, and let $\lambda>0$ such that $\int_{\mathbb B^n}\Phi\left(\frac{|f(z)|}{\lambda}\right)d\nu_\alpha(z)\leq 1.$ Using the fact that $\Phi_p(t):=\Phi(t^{1/p})$ is convex (see  Remark \ref{about-indices}), that $\frac{|f|^p}{\lambda^p}$ is subharmonic, and that the measures $\left(\frac{1-|z|^2}{|1-\langle z,w\rangle|^2}\right)^{n+1+\alpha}d\nu_\alpha(w)$ are probability measures (see \cite{R}), we obtain, for $z\in\mathbb B^n$,
\begin{eqnarray*}
    \Phi_p\left(\frac{|f(z)|^p}{\lambda^p}\right) &\leq & \int_{\mathbb B^n}\Phi_p\left(\frac{|f(w)|^p}{\lambda^p}\right)\left(\frac{1-|z|^2}{|1-\langle z,w\rangle|^2}\right)^{n+1+\alpha}d\nu_\alpha(w)\\
&\leq & \left( \frac{4}{1-|z|^2}\right)^{n+1+\alpha}\int_{\mathbb B^n}\Phi\left(\frac{|f(w)|}{\lambda}\right)d\nu_\alpha(w)\\
&\leq & C \frac{1}{\left(1-|z|^2\right)^{n+1+\alpha}}.
\end{eqnarray*}
Hence, for $z\in\mathbb B^n$, we have
\begin{eqnarray*}
    |f(z)|^p\leq C \lambda^p\left(\Phi_p\right)^{-1}\left(\frac{1}{\left(1-|z|^2\right)^{n+1+\alpha}}\right)\leq C \lambda^p\left(\Phi^{-1}\left(\frac{1}{\left(1-|z|^2\right)^{n+1+\alpha}}\right)\right)^p.
\end{eqnarray*}
From this, we have (\ref{eq:pointwise-concave}).
\end{proof}
We also provide norm estimates for bounded functions in $\mathcal A_{\alpha}^{\Phi}(\mathbb B^n)$. These are extension of the same type of result in \cite[Lemma 3.9]{LLQR}.
\begin{lemma}\label{bounded-concave}
Let $\alpha>-1$ and $\Phi\in \mathscr{L}_p$. For any bounded holomorphic function $f$ in $\mathbb B^n$, one has:
\begin{equation}\label{eq:bounded-concave}
\|f\|_{\alpha,\Phi}^{lux}\leq \frac{||f||_\infty}{\Phi^{-1}\left(\frac {||f||_\infty^p}{||f||_{\alpha,p}^p}\right)}.
\end{equation}
\end{lemma}
\begin{proof}
The proof follows exactly as in \cite{LLQR} where we use instead the fact that $\Phi_p(t):=\Phi(t^{1/p})$ is convex (see remark \ref{about-indices}).
\end{proof}
We also get the following estimates for bounded holomorphic functions in $\mathcal A_{\alpha}^{\Phi}(\mathbb B^n)$, when $\Phi\in \mathscr{U}^q$. The proof follows exactly as in \cite{LLQR}.
\begin{lemma}\label{bounded-convex}
Let $\alpha>-1$ and $\Phi\in \mathscr{U}^q$. Let $0<s<\infty$. Put $\Phi_s(t)=\Phi(t^s)$ and $\Phi^s(t)=(\Phi(t))^s$. For any bounded holomorphic function $f$ in $\mathbb B^n$, one has:
\begin{equation}\label{eq:bounded-convex1}
\|f\|_{\alpha,\Phi^s}^{lux}\leq \frac{||f||_\infty}{\Phi^{-1}\left(\frac {||f||_\infty}{||f||_{\alpha,s}}\right)}.
\end{equation}
and
\begin{equation}\label{eq:bounded-convex2}
\|f\|_{\alpha,\Phi_s}^{lux}\leq \left(\frac{||f||_\infty^s}{\Phi^{-1}\left(\frac {||f||_\infty^s}{||f||_{\alpha,s}^s}\right)}\right)^{1/s}.
\end{equation}
\end{lemma}

\section{Duality for large Bergman-Orlicz spaces}\label{section3}
The following Lemma generalizes the inequality (\ref{eq:intestimsmall}). As in the classical weighted Bergman spaces, this Lemma is crucial to characterize the dual space $\left(\mathcal A_{\alpha}^{\Phi}(\mathbb B^n)\right)^*$ of $\mathcal A_{\alpha}^{\Phi}(\mathbb B^n).$

\begin{lemma}\label{lem:embedconcave}
Let $\alpha>-1$, $\Phi\in \mathscr{L}_p$ and $\rho(t)=\frac{1}{t\Phi^{-1}(1/t)}$.  There is a constant $C>1$ such that for any $f\in \mathcal A_{\alpha}^{\Phi}(\mathbb B^n)$,
\begin{equation}\label{eq:embedconcave}
\int_{\mathbb B^n}|f(z)|\rho\left((1-|z|^2)^{n+1+\alpha}\right)d\nu_\alpha(z)\le C\|f\|_{\alpha,\Phi}^{lux}.
\end{equation}
\end{lemma}
\begin{proof}
The idea of the proof is an adaptation of the proof in the classical Bergman spaces and make uses of the pointwise estimate of functions in $\mathcal A_{\alpha}^{\Phi}(\mathbb B^n)$.
More precisely, let $f\in \mathcal A_{\alpha}^{\Phi}(\mathbb B^n)$ and $\lambda>0$ such that $\int_{\mathbb B^n}\Phi\left(\frac{|f(z)|}{\lambda}\right)d\nu_\alpha(z)\leq 1.$ We have
\begin{multline*}
\int_{\mathbb B^n}|f(z)|\rho\left((1-|z|^2)^{n+1+\alpha}\right)d\nu_\alpha(z) \\= \lambda \int_{\mathbb B^n}\Phi\left(\frac{|f(z)|}{\lambda}\right)\frac{\frac{|f(z)|}{\lambda}}{\Phi\left(\frac{|f(z)|}{\lambda}\right)} \rho\left((1-|z|^2)^{n+1+\alpha}\right)d\nu_\alpha(z).
\end{multline*}
Using (\ref{eq:pointwise-concave}), we have that, for $z\in\mathbb B^n$,
\begin{equation*}
    \frac{|f(z)|}{\lambda}\le C\Phi^{-1}\left(\frac 1{(1-|z|^2)^{n+1+\alpha}}\right)\left\|\frac f\lambda\right\|_{\alpha,\Phi}^{lux}\leq C\Phi^{-1}\left(\frac 1{(1-|z|^2)^{n+1+\alpha}}\right).
\end{equation*}
From this, using the fact that $t/\Phi(t)$ and $\Phi$ are non-decreasing, we have
\begin{multline*}
\int_{\mathbb B^n}|f(z)|\rho\left((1-|z|^2)^{n+1+\alpha}\right)d\nu_\alpha(z) \\ \leq  C\lambda \int_{\mathbb B^n}\Phi\left(\frac{|f(z)|}{\lambda}\right)\frac{\Phi^{-1}\left(\frac 1{(1-|z|^2)^{n+1+\alpha}}\right)}{\frac 1{(1-|z|^2)^{n+1+\alpha}}} \rho\left((1-|z|^2)^{n+1+\alpha}\right)d\nu_\alpha(z).
\end{multline*}
The definition of $\rho$ leads to
\begin{multline*}
\int_{\mathbb B^n}|f(z)|\rho\left((1-|z|^2)^{n+1+\alpha}\right)d\nu_\alpha(z) \leq  C\lambda \int_{\mathbb B^n}\Phi\left(\frac{|f(z)|}{\lambda}\right)d\nu_\alpha(z)\leq C\lambda.
\end{multline*}
The proof is complete.
\end{proof}

We next prove the duality result which also extends the classical duality for Bergman spaces with small exponents.
%
\begin{proof}[Proof of Theorem \ref{lem:dualitysmall}]
First using Lemma \ref{integrationparts} and Lemma \ref{lem:embedconcave}, we obtain that, for any $f\in \mathcal A_{\alpha}^{\Phi}(\mathbb B^n)$ and $g\in \Gamma_{\alpha, \rho}(\mathbb B^n)$,
\Beas
\left|\langle f,g\rangle_\alpha\right| &:=& \lim_{r\rightarrow 1}\left|\int_{\mathbb B^n}f(rz)\overline {g(z)}d\nu_\alpha(z)\right|\\
 &\le& C_{k,\alpha}\lim_{r\rightarrow 1}\int_{\mathbb
B^n}|f(rz)|\overline
|{M_{k}^\alpha g(z)}|(1-|z|^2)^{k}d\nu_\alpha (z)\\
&\leq&C \lim_{r\rightarrow 1}\int_{\mathbb
B^n}|f(rz)|\rho\left((1-|z|^2)^{n+1+\alpha}\right)d\nu_\alpha(z)\\ &\le&  C\|f\|_{\alpha,\Phi}^{lux}.
\Eeas
That is any $g\in \Gamma_{\alpha, \rho}(\mathbb B^n)$ defines a linear form on $\mathcal A_{\alpha}^{\Phi}(\mathbb B^n)$ under the pairing (\ref{eq:dualpairing}).

Now that any linear form on $\mathcal A_{\alpha}^{\Phi}(\mathbb B^n)$ is given by (\ref{eq:dualpairing}) can be justified by (\ref{eq:doubleinclus}) and the duality result for the usual weighted Bergman spaces (see \cite{KZ}). Thus we finish by proving that any element $g$ in $\left(\mathcal A_{\alpha}^{\Phi}(\mathbb B^n)\right)^*$ belongs to $\Gamma_{\alpha, \rho}(\mathbb B^n)$. Let $a\in\mathbb B^n$. We will test (\ref{eq:dualpairing}) with the function
$$f_a(z)=\Phi^{-1}\left(\frac{1}{(1-|a|^2)^{n+1+\alpha}}\right) \frac{(1-|a|^2)^{n+1+\alpha+k}}{\left(1-\langle z,a\rangle\right)^{n+1+\alpha+k}}$$
where $k$ is a fixed integer satisfying $k>(n+1+\alpha)(\frac{1}{p}-1)$.

Using Lemma \ref{bounded-concave}  and  Forelli-Rudin estimates (see for example \cite[Proposition 1.4.10]{R}), we see that $f_a$ is uniformly in $\mathcal A_{\alpha}^{\Phi}(\mathbb B^n)$. That is there exists a constant $C$, independant of $a$, such that $||f_a||_{\alpha,\Phi}^{lux}\leq C$.  We then get
\Beas
C &\ge& \left|\langle f_a,g\rangle_\alpha\right|
\\ &=& \lim_{r\rightarrow 1}\Phi^{-1}\left(\frac{1}{(1-|a|^2)^{n+1+\alpha}}\right)(1-|a|^2)^{n+1+\alpha+k}\left|\int_{\mathbb B^n}\frac{\overline {g(z)}}{\left(1-\langle rz,a\rangle\right)^{n+1+\alpha+k}}d\nu_\alpha(z)\right|\\
 &=& \lim_{r\rightarrow 1}\frac{(1-|a|^2)^k}{\rho\left((1-|a|^2)^{n+1+\alpha}\right)}\left|\int_{\mathbb B^n}\frac{\overline {g(z)}}{\left(1-\langle rz,a\rangle\right)^{n+1+\alpha+k}}d\nu_\alpha(z)\right| \\
&=& \frac 1{C_{k,\alpha}}\frac{(1-|a|^2)^k}{\rho\left((1-|a|^2)^{n+1+\alpha}\right)}\lim_{r\rightarrow 1}\left|M_k^{\alpha}g(ra)\right|.
\Eeas
So, there exists a constant $C$ such that for any $a\in \mathbb B^n$, $$\frac{(1-|a|^2)^k}{\rho\left((1-|a|^2)^{n+1+\alpha}\right)}|M_k^{\alpha}g(a)|\le C.$$ This is equivalent to the fact that $g\in \Gamma_{\alpha, \rho}(\mathbb B^n)$. The proof is complete.
\end{proof}

\section{Hankel operators between Bergman-Orlicz spaces}
We have gathered the results we need to study the boundedness of the small Hankel Operators between Bergman-Orlicz spaces.
\subsection{Boundedness of $h_b:A_\alpha^{\Phi}(\mathbb B^n) \rightarrow A_\alpha^{\Phi}(\mathbb B^n)$, $\Phi\in \mathscr{U}^q$}

In this subsection, we consider boundedness of Hankel operators $h_b$ on the Bergman-Orlicz spaces $A_\alpha^{\Phi}(\mathbb B^n)$
for $\Phi$ a convex growth function in the class $\mathscr{U}^q$. We start by considering a general weak factorization of weighted
Bergman spaces with small exponents.

\begin{proposition}\label{prop:weakfactsmallexpo}
Let $\alpha>-1$, $0<p\le 1$ and $\Phi\in \mathscr U^q$. Denote by $\Psi$ the complementary function of $\Phi$. If we define $\Phi_p$ by $\Phi_p(t)=\Phi(t^p)$, then every function $f\in \mathcal A_\alpha^p(\mathbb B^n)$ admits the following representation
\begin{equation}\label{eq:weakfactsmallexpo}f(z)=\sum_{j}f_j(z)g_j(z),\,\,\,z\in \mathbb B^n,
\end{equation}
where each $f_j$ is in $\mathcal A_{\alpha}^{\Phi_p}(\mathbb B^n)$ and each $g_j$ is in $\mathcal A_{\alpha}^{\Psi_p}(\mathbb B^n)$. Moreover, we have
$$\sum_{j}\|f_j\|^{lux}_{\alpha,\Phi_p}\|g_j\|^{lux}_{\alpha,\Psi_p}\le C\|f\|_{p,\alpha},$$
where $C$ is a positive constant independent of $f$.
\end{proposition}
\begin{proof}
First, let us recall with \cite[Theorem 2.30]{KZ} that, for $b>(n+1+\alpha)/p$, there exists a sequence $\{a_j\}$ in $\mathbb B^n$ such that every $f\in \mathcal {A}_\alpha^p(\mathbb B^n)$ admits the following representation
$$f(z)=\sum_jc_j\frac{\left(1-|a_j|^2\right)^{b-(n+1+\alpha)/p}}{\left(1-\langle z,a_j\rangle\right)^{b}},$$
where $\{c_j\}$ belongs to the sequence space $l^p$ and the series converges in the norm topology of $\mathcal {A}_\alpha^p(\mathbb B^n)$.
Now take, for non zero $c_j$,
$$f_j(z)=c_j^{p/q}\left(\frac{\left(1-|a_j|^2\right)^{b-(n+1+\alpha)/p}}{\left(1-\langle z,a_j\rangle\right)^{b}}\right)^{p/q}$$
and
$$g_j(z)=c_j^{p/r}\left(\frac{\left(1-|a_j|^2\right)^{b-(n+1+\alpha)/p}}{\left(1-\langle z,a_j\rangle\right)^{b}}\right)^{p/r}$$
where $\frac{1}{p}=\frac{1}{q}+\frac{1}{r}$. It is clear that (\ref{eq:weakfactsmallexpo}) holds.
Using Lemma \ref{bounded-convex}  and  Forelli-Rudin estimates (see for example \cite[Proposition 1.4.10]{R}) with $b$ large enough, we have
\begin{eqnarray}\label{eq:estimate1}
    \|f_j\|_{\alpha,\Phi_p}^{lux}\leq C |c_j|^{p/q}\frac{\left(1-|a_j|^2\right)^{-(n+1+\alpha)/q}}
{\left(\Phi^{-1}\left(\frac{1}{(1-|a_j|^2)^{n+1+\alpha}}\right)\right)^{1/p}}
\end{eqnarray}
and
\begin{eqnarray}\label{eq:estimate2}
    \|g_j\|_{\alpha,\Psi_p}^{lux}\leq C |c_j|^{p/q}\frac{\left(1-|a_j|^2\right)^{-(n+1+\alpha)/r}}
{\left(\Psi^{-1}\left(\frac{1}{(1-|a_j|^2)^{n+1+\alpha}}\right)\right)^{1/p}}.
\end{eqnarray}
Now, using (\ref{eq:estimate1}), (\ref{eq:estimate2}) and the fact that $\Phi^{-1}(t)\Psi^{-1}(t)\backsimeq t$, we have
\begin{multline*}   \left(\sum_{j}\|f_j\|^{lux}_{\alpha,\Phi_p}\|g_j\|^{lux}_{\alpha,\Psi_p}\right)^p  \le \\ C\left(\sum_{j}|c_j|^{p/q}|c_j|^{p/r}\frac{\left(1-|a_j|^2\right)^{-(n+1+\alpha)/p}}
{\left(\Phi^{-1}\left(\frac{1}{(1-|a_j|^2)^{n+1+\alpha}}\right)
\Psi^{-1}\left(\frac{1}{(1-|a_j|^2)^{n+1+\alpha}}\right)\right)^{1/p}}\right)^p \\
\leq C \left(\sum_j|c_j|\right)^p\leq C\sum_j|c_j|^{p}\leq C \|f\|_{p,\alpha}^p.
\end{multline*}
This finishes the proof.
\end{proof}

Using the weak factorization with $p=1$, we can know show that, as in the classical case, for convex growth function $\Phi$, the small Hankel operator on $\mathcal A_{\alpha}^{\Phi}(\mathbb B^n)$ is bounded if and only if the symbol lies in the Bloch space.
\begin{theorem}
Let $\Phi\in \mathscr U^q$  such that $\Phi$ satisfies the Dini condition (\ref{dinicondition}), and $\alpha>-1$. Then the Hankel operator $h_b$ extends into a bounded operator on $\mathcal A_{\alpha}^{\Phi}(\mathbb B^n)$ if and only if $b\in \mathcal {B}$.
\end{theorem}
\begin{proof}
First we recall that since $\Phi$ satisfies (\ref{dinicondition}), then the dual of $\mathcal A_{\alpha}^{\Phi}(\mathbb B^n)$ coincides with $\mathcal A_{\alpha}^{\Psi}(\mathbb B^n)$ with $\Psi$ the complementary function of $\Phi$. Since $\Phi^{-1}(t)\Psi^{-1}(t)\simeq t$,
by Proposition \ref{prop:volbergtolokonnikov}, for any $f\in \mathcal A_{\alpha}^{\Phi}(\mathbb B^n)$ and $g\in \mathcal A_{\alpha}^{\Psi}(\mathbb B^n)$, the product $fg$ is in $\mathcal A_{\alpha}^{1}(\mathbb B^n)$ which has dual space the Bloch space $\mathcal B$. Consequently, there is a constant $C>0$ such that, for $f\in \mathcal A_{\alpha}^{\Phi}(\mathbb B^n)$ and $g\in \mathcal A_{\alpha}^{\Psi}(\mathbb B^n)$
$$|\langle  h_b(f),g\rangle_\alpha|=|\langle b,f g
\rangle_\alpha|\le C||b||_{\mathcal B}||fg||_{1,\alpha}\le C||b||_{\mathcal B}||f||^{lux}_{\alpha,\Phi}||g||^{lux}_{\alpha,\Psi}.$$ We conclude that if $b\in
\mathcal B$, then $h_b$ is bounded
from $\mathcal A_\alpha^{\Phi}(\mathbb B^n)$ into itself with $||h_b||\leq C||b||_{\mathcal B}.$

Conversely, we suppose that $h_b$ extends into a bounded
operator on $\mathcal A_{\alpha}^{\Phi}(\mathbb B^n)$. To prove that the symbol $b\in\mathcal B$, we only need to prove that there is a constant $C>0$ such that for any $f\in \mathcal A_{\alpha}^{1}(\mathbb B^n)$,
\begin{equation}\label{eq:necesssmall}
|\langle b,f\rangle_\alpha |\le C\|f\|_{1,\alpha}.
\end{equation}
From Proposition \ref{prop:weakfactsmallexpo} we have that any $f\in \mathcal A_{\alpha}^1(\mathbb B^n)$ can be written as $f=\sum_{j}f_jg_j$
with $\sum_{j}\|f_j\|^{lux}_{\alpha,\Phi}\|g_j\|^{lux}_{\alpha,\Psi}\le C\|f\|_{1,\alpha}$. It follows that
\Beas
|\langle b,f\rangle_\alpha| &\le& \sum_j |\langle b,f_jg_j\rangle_\alpha|
 = \sum_j |\langle h_b(f_j),g_j\rangle_\alpha|\\
&\le& \sum_{j}\|h_b(f_j)\|^{lux}_{\alpha,\Phi}\|g_j\|^{lux}_{\alpha,\Psi}\\ &\le& \|h_b\|\sum_{j}\|f_j\|^{lux}_{\alpha,\Phi}\|g_j\|^{lux}_{\alpha,\Psi}\\ &\le& C\|h_b\|\|f\|_{1,\alpha}.
\Eeas
Thus, we have (\ref{eq:necesssmall})  and this  complete the proof.
\end{proof}

\subsection{Boundedness of $h_b$: $\mathcal A_\alpha^{\Phi_1}(\mathbb
 B^n)\rightarrow \mathcal A_\alpha^{\Phi_2}(\mathbb
 B^n)$;  $(\Phi_1,\Phi_2)\in \mathscr{U}^q\times \mathscr{L}_p $}

 We start this section by defining the space $\mathcal {A}_{weak}^1(\mathbb B^n)$.
 \begin{definition} The space $\mathcal {A}_{weak}^1(\mathbb B^n)$ consists of all holomorphic functions $f$ such that, for some constant $C>0$, we have
$$\lambda \nu_\alpha \left(\{z\in \mathbb B^n:|f(z)|>\lambda\}\right)\le C\,\,\,\textrm{for any}\,\,\, \lambda>0.$$
\end{definition}
This becomes a Banach space under the following norm
$$||f||_{1,weak}=\sup_{\lambda>0}\lambda \nu_\alpha \left(\{z\in \mathbb B^n:|f(z)|>\lambda\}\right).$$
We observe that from the above definition and Proposition \ref{prop:weakboundedness} we have that the Bergman projection $P_\alpha$ is
bounded from $L_\alpha^1(\mathbb B^n)$ to $\mathcal {A}_{weak}^1(\mathbb B^n)$.

We have the following embedding result.
\begin{proposition}\label{h1faibleorlicz}
Let $\Phi\in \mathscr{L}_p$, $\alpha >-1$. Suppose that $\Phi$ satisfies the Dini's condition
\begin{equation}\label{Dinicondition1}
\int_1^\infty\frac{\Phi(t)}{t^2}dt<\infty.
\end{equation}
Then $\mathcal {A}_{weak}^1(\mathbb B^n)$ embeds continuously in $\mathcal A_\alpha^{\Phi}(\mathbb
 B^n)$
\end{proposition}
\begin{proof}
It is enough to prove that for any $f\in \mathcal {A}_{weak}^1(\mathbb B^n)$, there exists $C>0$ such that
  $$\int_{\mathbb B^n}\Phi(|f|)(z)d\nu_\alpha (z)\le C.$$
We have
\begin{eqnarray*}
\int_{\mathbb B^n}\Phi(|f|)(z)d\nu_\alpha (z) &=& \int_0^\infty \nu_\alpha \left(\{z\in \mathbb B^n:|f(z)|>\lambda\}\right)\Phi^\prime(\lambda)d\lambda\\
&=& I+J,
\end{eqnarray*}
where $$I=\int_0^1 \nu_\alpha \left(\{z\in \mathbb B^n:|f(z)|>\lambda\}\right)\Phi^\prime(\lambda)d\lambda$$
and $$J=\int_1^\infty \nu_\alpha \left(\{z\in \mathbb B^n:|f(z)|>\lambda\}\right)\Phi^\prime(\lambda)d\lambda.$$
It is clear that
$$I\le \nu_\alpha (\mathbb B^n)\int_0^1\Phi^\prime(\lambda)d\lambda=\Phi(1).$$
To estimate the integral $J$, we use the definition of $\mathcal {A}_{weak}^1(\mathbb B^n)$, the fact that $\Phi'(t)\backsimeq \frac{\Phi(t)}{t}$ and
$\Phi$ satisfies the Dini's condition (\ref{Dinicondition1}). We obtain
\begin{eqnarray*}
J &=& \int_1^\infty \nu_\alpha \left(\{z\in \mathbb B^n:|f(z)|>\lambda\}\right)\Phi^\prime(\lambda)d\lambda\\
 &\le& ||f||_{1,weak}\int_1^\infty \frac{\Phi'(\lambda)}{\lambda}d\lambda \backsimeq C ||f||_{1,weak} \int_1^\infty \frac{\Phi(\lambda)}{\lambda^2}d\lambda\le C ||f||_{1,weak}.
\end{eqnarray*}
The proof is complete.
\end{proof}

Under the Dini's condition (\ref{Dinicondition1}), we easly obtain the boundedness criteria for the small Hankel operator $h_b$ from  $\mathcal A_\alpha^{\Phi_1}(\mathbb
 B^n)$ into $\mathcal A_\alpha^{\Phi_2}(\mathbb
 B^n)$ when  $(\Phi_1,\Phi_2)\in \mathscr{U}^q\times \mathscr{L}_p $. This is a  generalization of the case $h_b:\mathcal {A}_\alpha^p(\mathbb B^n)\rightarrow \mathcal A_\alpha^q(\mathbb B^n)$ with $1< p<\infty$ and $0<q<1$.
\begin{theorem}\label{theo:casavecpertesimple}
Let $\Phi_1\in \mathscr{U}^q$ and $\Phi_2\in \mathscr{L}_p$, $\alpha>-1$. Let $\Psi_1$ be the complementary function of $\Phi_1$ and, suppose that $\Phi_1$ satisfies the Dini's condition (\ref{dinicondition}) while $\Phi_2$
satisfies (\ref{Dinicondition1}). Then $h_b$ extends as a bounded operator from $\mathcal A_\alpha^{\Phi_1}(\mathbb
 B^n)$ to $\mathcal A_\alpha^{\Phi_2}(\mathbb B^n)$ if and only if $b\in \mathcal A_\alpha^{\Psi_1}(\mathbb
 B^n)$.
\end{theorem}
\begin{proof}
We start by proving the necessity. Suppose that $h_b$ is bounded from $\mathcal A_\alpha^{\Phi_1}(\mathbb
 B^n)$ to $\mathcal A_\alpha^{\Phi_2}(\mathbb B^n)$. Then for any $f\in \mathcal A_\alpha^{\Phi_1}(\mathbb
 B^n)$, we have $$\left|\int_{\mathbb B^n}b(\xi)\overline {f(\xi)}d\nu_\alpha(\xi)\right|=|h_bf(0)|\le C||h_b(f)||^{lux}_{\alpha,\Phi_2}\le C||h_b||||f||^{lux}_{\alpha,\Phi_1}.$$
  We have used the fact that $\mathcal A_\alpha^{\Phi_2}(\mathbb B^n)$ is continuously contained in $\mathcal A_\alpha^p(\mathbb B^n)$, and the evaluation at $0$ is bounded on this space. It follows that $b$ belongs to the dual space of $\mathcal A_\alpha^{\Phi_1}(\mathbb B^n)$
 that is $b\in \mathcal A_\alpha^{\Psi_1}(\mathbb B^n)$.

 Conversely, if $b\in \mathcal A_\alpha^{\Psi_1}(\mathbb B^n)$, then for any $f\in \mathcal A_\alpha^{\Phi_1}(\mathbb
 B^n)$, the product $b\overline {f}$ is in $L_\alpha^1(\mathbb B^n)$ by Proposition \ref{prop:volbergtolokonnikov}. Thus, $h_b(f):=P_\alpha (b\overline {f})$ is in $\mathcal {A}_{weak}^1(\mathbb B^n)$ and consequently in
 $\mathcal A_\alpha^{\Phi_2}(\mathbb B^n)$ by Proposition \ref{h1faibleorlicz}. The proof is complete.
\end{proof}

\subsection{Boundedness of $h_b$: $\mathcal A_\alpha^{\Phi_1}(\mathbb
 B^n)\rightarrow \mathcal A_\alpha^{\Phi_2}(\mathbb
 B^n)$;  $(\Phi_1,\Phi_2)\in \mathscr{L}_p\times \mathscr{L}_p $}

Let us start this section by the following result.
\btheo\label{theo:h_bfacto}
Let $\Phi_1\in \mathscr{L}_p$, $\alpha>-1$ and $\Phi_2\in \mathscr{L}_p\cup \mathscr{U}^q$. If $h_b$ extends into a bounded operator from $\mathcal A_\alpha^{\Phi_1}(\mathbb B^n)$ into $\mathcal A_\alpha^{\Phi_2}(\mathbb B^n)$, then the symbol $b$ belongs to $\Gamma_{\alpha,\rho_1}(\mathbb B^n)$ with
$\rho_1(t)=\frac{1}{t\Phi_1(\frac{1}{t})}$. Conversely, if $b\in \Gamma_{\alpha,\rho_1}(\mathbb B^n)$, then there exists a bounded operator from $\mathcal A_\alpha^{\Phi_1}(\mathbb B^n)$ into $L_\alpha^1(\mathbb B^n)$ which we note $T_b$ such that $h_b=P_\alpha T_b$.
\etheo
\begin{proof}
That the boundedness of $h_b$ from $\mathcal A_\alpha^{\Phi_1}(\mathbb B^n)$ into $\mathcal A_\alpha^{\Phi_2}(\mathbb B^n)$ implies that $b$ belongs to
$\Gamma_{\alpha,\rho_1}(\mathbb B^n)=\left(\mathcal A_\alpha^{\Phi_1}(\mathbb B^n)\right)^*$ follows as in the first part of the proof of Theorem \ref{theo:casavecpertesimple}.
\vskip .2cm
It is easy to check that $h_b=P_\alpha T_b$ (see for example \cite{BL}) with $$T_bf(z)=M_k^\alpha b(z)(1-|z|^2)^k\overline {f(z)},$$
where $k>(n+1+\alpha)\left(\frac 1p -1\right)$.
Recalling that $b\in \Gamma_{\alpha,\rho_1}(\mathbb B^n)$ is equivalent in saying that for some constant $C>0$,
$$|M_k^\alpha b(z)|(1-|z|^2)^k\le C\rho_1\left((1-|z|^2)^{n+1+\alpha}\right),\,\,\,\textrm{for all}\,\,\, z\in \mathbb B^n,$$
and using Lemma \ref{lem:embedconcave}, we easily get
\Beas
\int_{\mathbb B^n}|T_bf(z)|d\nu_\alpha(z) &\le& C\|b\|_{\Gamma_{\alpha,\rho_1}}\int_{\mathbb B^n}|f(z)|\rho_1\left((1-|z|)^{n+1+\alpha}\right)d\nu_\alpha(z)\\ &\le&
C\|b\|_{\Gamma_{\alpha,\rho_1}}\|f\|_{\alpha,\Phi_1}.
\Eeas
\end{proof}

As a corollary, we obtain
\bcor
Let $\Phi_1, \Phi_2\in \mathscr{L}_p$, $\alpha>-1$. Suppose moreover that $\Phi_2$ satisfies the Dini condition (\ref{Dinicondition1}). Then $h_b$ extends into a bounded operator from $\mathcal A_\alpha^{\Phi_1}(\mathbb B^n)$ into $\mathcal A_\alpha^{\Phi_2}(\mathbb B^n)$ if and only if the symbol $b$ belongs to $\Gamma_{\alpha,\rho_1}(\mathbb B^n)$ with
$\rho_1(t)=\frac{1}{t\Phi_1(\frac{1}{t})}$.
\ecor
\begin{proof}
First we observe that the necessity is given by Theorem \ref{theo:h_bfacto}. From the same theorem, we have that for any $f\in \mathcal A_\alpha^{\Phi_1}(\mathbb B^n)$, $T_b(f)\in L_\alpha^1(\mathbb B^n)$ with $T_bf(z)=M_k^\alpha b(z)(1-|z|^2)^k\overline {f(z)}$. As $h_b=P_\alpha T_b$,
we obtain that $h_b(f)\in \mathcal A_{weak}^1(\mathbb B^n)$. The conclusion follows now from that as $\Phi_2$ satisfies (\ref{Dinicondition1}), $\mathcal A_{weak}^1(\mathbb B^n)$ embeds continuously into $\mathcal A_\alpha^{\Phi_2}(\mathbb B^n)$.
\end{proof}

\subsection{Boundedness of $h_b$: $\mathcal A_\alpha^{\Phi}(\mathbb
 B^n)\rightarrow \mathcal A_\alpha^{1}(\mathbb
 B^n)$;  $\Phi\in \mathscr{L}_p $}
We start this section by observing that, for $�\in\mathbb B^n$, the following function $$f_a(z)=\Phi^{-1}\left(\left(\frac{1}{(1-|a|^2)^{n+1+\alpha}}\right)\right)\frac{(1-|a|^2)^{n+1+\alpha+k}}{(1-\langle z,a\rangle)^{n+1+\alpha+k}}\log\frac{1-\langle z,a\rangle}{1-|a|^2},\,\,\,z\in \mathbb B^n $$ is uniformly in $\mathcal A_\alpha^{\Phi}(\mathbb
 B^n)$ for $k$ an integer such that $k>(n+1+\alpha)(\frac{1}{p}-1)$.

 To see this recall that for any $\varepsilon>0$, there is a constant $C>0$ such that $\log^+x\le Cx^\varepsilon$. It follows that
 $$|f_a(z)|\le C\Phi^{-1}\left(\frac{1}{(1-|a|^2)^{n+1+\alpha}}\right)\frac{(1-|a|^2)^{n+1+\alpha+k-\varepsilon}}{|1-\langle z,a\rangle|^{n+1+\alpha+k-\varepsilon}},$$
for $\varepsilon$ small enough. We have also used the fact that, for $a,z\in\mathbb B^n,$
$$
\left|\frac{1-\langle z,a\rangle}{1-|a|^2}\right|\geq \frac 12.
$$
Choosing $0<\varepsilon<k-(n+1+\alpha)(\frac{1}{p}-1)$, we write
\begin{multline*}
\int_{\mathbb B^n}\Phi(|f_a(z)|)d\nu_\alpha(z)=\\
\int_{\frac{(1-|a|^2)}{|1-\langle z,a\rangle|}\le 1}\Phi(|f_a(z)|)d\nu_\alpha(z)+\int_{\frac{(1-|a|^2)}{|1-\langle z,a\rangle|}> 1}\Phi(|f_a(z)|)d\nu_\alpha(z)=I_1+I_2.
\end{multline*}

Using that $\Phi$ is of lower type $p$ and \cite[Proposition 1.4.10]{R}, we obtain
\Beas
I_1 &=& \int_{\frac{(1-|a|^2)}{|1-\langle z,a\rangle|}\le 1}\Phi(|f_a(z)|)d\nu_\alpha(z)\\ &\le& C\frac{2}{(1-|a|^2)^{n+1+\alpha}}\int_{\mathbb B^n}\frac{(1-|a|^2)^{(n+1+\alpha+k-\varepsilon)p}}{|1-\langle z,a\rangle|^{(n+1+\alpha+k-\varepsilon)p}}d\nu_\alpha(z)\\ &\le& C.
\Eeas

For the second integral, we use the fact that the function $\frac{\Phi^{-1}(t)}{t}$ is non-decreasing on $(0,\infty)$, and \cite[Proposition 1.4.10]{R} to obtain
\Beas
I_2 &=& \int_{\frac{(1-|a|^2)}{|1-\langle z,a\rangle|}> 1}\Phi(|f_a(z)|)d\nu_\alpha(z)\\ &\le& C\frac{2}{(1-|a|^2)^{n+1+\alpha}}\int_{\mathbb B^n}\frac{(1-|a|^2)^{n+1+\alpha+k-\varepsilon}}{|1-\langle z,a\rangle|^{n+1+\alpha+k-\varepsilon}}d\nu_\alpha(z)\\ &\le& C.
\Eeas

The following result holds.
\btheo\label{theo:casversA1}
Let $\Phi\in \mathscr{L}_p$, $\alpha>-1$. Then the Hankel operator $h_b$ extends into a bounded operator from $\mathcal A_\alpha^{\Phi}(\mathbb B^n)$ into $\mathcal A_\alpha^{1}(\mathbb B^n)$ if and only if the symbol $b$ belongs to $L\Gamma_{\alpha,\rho}(\mathbb B^n)$ with
$\rho(t)=\frac{1}{t\Phi(\frac{1}{t})}$.
\etheo
\begin{proof}
First, we suppose that the function $b\in L\Gamma_{\alpha,\rho}(\mathbb B^n).$ That is
$b\in \mathcal H(\mathbb B^n)$  and satisfies, for some $k$ large enough,  the condition
$$|R^kb(z)|\le
C(1-|z|^2)^{-k}\rho\left((1-|z|^2)^{n+1+\alpha}\right)\left(\log\frac{1}{1-|z|^2}\right)^{-1},$$
with $C>0$ an absolute constant.
We prove that $h_b$ is bounded from $\mathcal A_\alpha^{\Phi}(\mathbb B^n)$ into
$\mathcal A_\alpha^1(\mathbb B^n)$. Using Lemma \ref{integrationparts}, we have
\Beas
h_b(f)(z)=P_\alpha(b\overline f)(z) &=& \int_{\mathbb
B^n}b(w)\overline {f(w)} K_\alpha(z,w)d\nu_\alpha(w)\\
&=& C_{k,\alpha}\int_{\mathbb B^n}K_\alpha(z,w)M_k^\alpha b(w)\overline {f(w)}
(1-|w|^2)^{k}d\nu_\alpha(w).\Eeas Now, by
 \cite[Proposition 1.4.10]{R}, we have $$\int_{\mathbb
B^n}|K_\alpha(z,w)|d\nu_\alpha(z)\backsimeq
\log\frac{1}{1-|w|^2}.$$
 Combining  these facts with the inequality  (\ref{eq:embedconcave}), we obtain
 \Beas
\quad &&\int_{\mathbb B^n}|h_bf(z)|d\nu_\alpha(z) \\
&\le&
C\int_{\mathbb B^n}\int_{\mathbb
B^n}|K_\alpha(z,w)||M_k^\alpha b(w)||f(w)|(1-|w|^2)^{k}d\nu_\alpha (w)d\nu_\alpha(z)\\
&\le& C\int_{\mathbb
B^n}|M_k^\alpha b(w)|\left(\log\frac{1}{1-|w|^2}\right)(1-|w|^2)^{k}|f(w)|d\nu_\alpha (w)\\
&\le& C ||b||_{L\Gamma_{\alpha,\rho}}\int_{\mathbb
B^n}|f(w)|\rho\left((1-|z|^2)^{n+1+\alpha}\right)d\nu_\alpha (w)\\
 &\le&
C ||b||_{L\Gamma_{\alpha,\rho}}||f||^{lux}_{\alpha,\Phi} .\Eeas
This complete the first part of the proof.

Conversely, if $h_b$ is bounded from
$\mathcal A_\alpha^\Phi (\mathbb B^n)$ into $\mathcal A_\alpha^1(\mathbb
B^n)$, then
we have for every $f\in \mathcal A_\alpha^\Phi (\mathbb B^n)$ and $g\in
\mathcal B=(\mathcal A_\alpha^1(\mathbb B^n))^*$,
\begin{equation}\label{dualineq}
|\langle
h_b(f),g\rangle_\alpha|=|\langle b, f g \rangle_\alpha|\le
C||h_b||||f||^{lux}_{\alpha,\Phi}||g||_{\mathcal B}.\end{equation}

We will apply the inequality (\ref{dualineq}) to $f$ and $g$, with
$$f(z)=f_w(z)=\Phi^{-1}\left(\frac{1}{(1-|w|^2)^{n+1+\alpha}}\right)\frac{(1-|w|^2)^{n+1+\alpha+k}}{(1-\langle z,w\rangle)^{n+1+\alpha+k}}$$ and  $$g(z)=\log(1-\langle z,w
\rangle)$$ where $k$ is an integer with
$k>(n+1+\alpha)(\frac{1}{p}-1)$.

We have seen that $f$ is uniformly in $\mathcal A_\alpha^\Phi (\mathbb
B^n)$ and it is well known that $g$ is uniformly in
$\mathcal B$. It follows that

\begin{multline*}
||h_b|| \ge C\Phi^{-1}\left(\frac{1}{(1-|w|^2)^{n+1+\alpha}}\right)(1-|w|^2)^{n+1+\alpha+k}\\ \left|\lim_{r\rightarrow 1}\int_{\mathbb B^n}\frac{b(z)}{\left(1-\langle w,rz
\rangle\right)^{n+1+\alpha+k}}\log(1-\langle w,rz \rangle)d\nu_\alpha(z) \right|\\ =C\frac{(1-|w|^2)^k}{\rho\left((1-|w|^2)^{n+1+\alpha}\right)}\left|\lim_{r\rightarrow 1}\int_{\mathbb B^n}\frac{b(z)}{\left(1-\langle w,rz
\rangle\right)^{n+1+\alpha+k}}\log(1-\langle
w,rz \rangle)d\nu_\alpha(z) \right|.
\end{multline*}

This  is equivalent to
\begin{multline*}
\left|\lim_{r\rightarrow 1}\int_{\mathbb B^n}\frac{(1-|w|^2)^k}{\rho\left((1-|w|)^{n+1+\alpha}\right)}\frac{b(z)}{\left(1-\langle w,rz
\rangle\right)^{n+1+\alpha+k}}\log(1-|w|^2)d\nu_\alpha(z)+\right. \\
\left.
\lim_{r\rightarrow 1}\int_{\mathbb B^n}b(z)\overline {h(rz)}d\nu_\alpha(z)
\right|\le C ||h_b||
\end{multline*}
where
$$h(z)=\Phi^{-1}\left(\frac{1}{(1-|w|^2)^{n+1+\alpha}}\right)\frac{(1-|w|^2)^{n+1+\alpha+k}}{(1-\langle z,w\rangle)^{n+1+\alpha+k}}\log\left(\frac{1-\langle z,w
\rangle}{1-|w|^2}\right).$$  We have seen at the beginning of this section that  $h$ is
uniformly in $\mathcal A_\alpha^\Phi(\mathbb B^n)$. That is
 $$||h||^{lux}_{\alpha,\Phi}\leq C.$$ It follows using (\ref{dualineq}) with $f=h$ and $g=1$ that
 $$\left|\lim_{r\rightarrow 1}\int_{\mathbb
B^n}b(z)\overline {h(rz)}d\nu_\alpha(z)\right|\le C||h_b||.$$
We deduce that
$$\left|\lim_{r\rightarrow 1}\int_{\mathbb B^n}\frac{(1-|w|^2)^k}{\rho\left((1-|w|^2)^{n+1+\alpha}\right)}\frac{b(z)}{\left(1-\langle w,rz
\rangle\right)^{n+1+\alpha+k}}\log(1-|w|^2)d\nu_\alpha(z)\right|\le C||h_b||$$ or equivalently
\Beas \left|\lim_{r\rightarrow 1}\int_{\mathbb B^n}\frac{b(z)}{\left(1-\langle w,rz
\rangle\right)^{n+1+\alpha+k}}d\nu_\alpha(z) \right| &\le&
C||h_b||(1-|w|^2)^{-k}\\ & & \rho\left((1-|w|^2)^{n+1+\alpha}\right)\left(\log\frac{1}{1-|w|^2}\right)^{-1}.
\Eeas
That is, for $w\in
\mathbb B^n$
 $$|M_{k}^\alpha b(w)|\le
C||h_b||(1-|w|^2)^{-k}\rho\left((1-|w|^2)^{n+1+\alpha}\right)\left(\log\frac{1}{1-|w|^2}\right)^{-1}.$$ The proof is complete.
\end{proof}

\subsection{Boundedness of $h_b$: $\mathcal A_\alpha^{\Phi_1}(\mathbb
 B^n)\rightarrow \mathcal A_\alpha^{\Phi_2}(\mathbb
 B^n)$;  $(\Phi_1,\Phi_2)\in \mathscr{L}_p\times \mathscr{U}^q $}
 The following result extends the classical case $h_b:\mathcal A_\alpha^{p}(\mathbb
 B^n)\rightarrow \mathcal A_\alpha^{q}(\mathbb
 B^n)$, $0<p<1$ and $q>1$.
 \begin{theorem}\label{theo:concavetoconvex}
Let $\Phi_1\in \mathscr{L}_p$ and $\Phi_2\in \mathscr{U}^q$, $\rho_i(t)=\frac{1}{t\Phi_i^{-1}(1/t)}$ and assume that $\Phi_2$ satisfies the Dini condition (\ref{dinicondition}). Then   the Hankel operator $h_b$ extends into a bounded operator from $\mathcal A_\alpha^{\Phi_1}(\mathbb
 B^n)$ into $\mathcal A_\alpha^{\Phi_2}(\mathbb
 B^n)$ if and only if its symbol $b$ belongs to  $ \Gamma_{\alpha,\rho}(\mathbb B^n)= (\mathcal A^{\Phi}(\mathbb
 B^n))^*,$ where
\begin{equation*}
\rho=\rho_{\Phi}:=\frac{\rho_1}{\rho_2}.
\end{equation*}
\end{theorem}
\begin{proof}
We start by proving the sufficiency of the condition for the
boundedness. Let
$\Phi_j$, $j=1,2$ be as in the hypothesis and $\rho=\rho_{\Phi}:=\frac{\rho_1}{\rho_2}$. Denoting by $\Psi_2$ the complementary function of $\Phi_2$, then as $\Phi^{-1}(t)=\Phi_1^{-1}(t)\Psi_2^{-1}(t)$, Proposition \ref{prop:volbergtolokonnikov}
gives that $fg\in \mathcal A_\alpha^{\Phi}(\mathbb
 B^n)$ for any $f\in \mathcal A_\alpha^{\Phi_1}(\mathbb
 B^n)$ and $g\in \mathcal A_\alpha^{\Psi_2}(\mathbb
 B^n)$. Moreover, the dual space of $\mathcal A_\alpha^{\Phi}(\mathbb
 B^n)$ coincides with $\Gamma_{\alpha, \rho}$ since $\Phi\in \mathscr{L}_r$ for some $0<r\le p$ (see Lemma \ref{stabilityoflowertypeclass}). It
follows that there exists a positive constant $C$
such that
$$|\langle  h_b(f),g\rangle_\alpha|=|\langle b,f g
\rangle_\alpha|\le C||b||_{\Gamma_{\alpha,\rho}}||fg||^{lux}_{\alpha,\Phi}\le  C||b||_{\Gamma_{\alpha,\rho}}||f||^{lux}_{\alpha,\Phi_1}||g||^{lux}_{\alpha,\Psi_2}.$$ We conclude that if $b\in
\Gamma_{\alpha,\rho}(\mathbb B^n)$, then $h_b$ is bounded
from $\mathcal A_\alpha^{\Phi_1}(\mathbb B^n)$ into $\mathcal A_\alpha^{\Phi_2}(\mathbb
B^n)$ with $||h_b||\leq C||b||_{\Gamma_{\alpha.\rho}}.$

Conversely, suppose that $h_b$ extends into a bounded
operator from $\mathcal A_\alpha^{\Phi_1}(\mathbb B^n)$ into $\mathcal A_\alpha^{\Phi_2}(\mathbb B^n)$. Then as in (\ref{dualineq}), we have
\begin{equation}\label{dualityhankel}
|\langle h_bf,g\rangle_\alpha|=\left|\lim_{r\rightarrow 1}\int_{\mathbb B^n}b(z)\overline
{f(rz)g(rz)}d\nu_\alpha(z)\right|\le
C||h_b||||f||^{lux}_{\alpha,\Phi_1}||g||^{lux}_{\alpha,\Psi_2}.
\end{equation}
 Let $w\in \mathbb B^n,$ we apply the above inequality to
$$f(z)=f_w(z)=\Phi_1^{-1}\left(\frac{1}{\left(1-|w|^2\right)^{n+1+\alpha}}\right)\frac{(1-|w|^2)^{k-1}}{(1-\langle z,w \rangle)^{k-1}},$$ and
$$g(z)=\Psi_2^{-1}\left(\frac{1}{\left(1-|w|^2\right)^{n+1+\alpha}}\right)\frac{(1-|w|^2)^{n+2+\alpha}}{(1-\langle z,w \rangle)^{n+2+\alpha}},$$
 with $k>\frac{n+1+\alpha}{p}+1$. Using Lemma \ref{bounded-concave} and Lemma \ref{bounded-convex}, one easily verifies that $f$ and $g$ are  uniformly in $\mathcal A_\alpha^{\Phi_1}(\mathbb B^n)$ and
$\mathcal A_\alpha^{\Psi_2}(\mathbb B^n)$ respectively. Hence
\begin{multline*}
 |\langle h_b(f),g\rangle_\alpha| \\
\simeq \Phi^{-1}\left(\frac{1}{\left(1-|w|^2\right)^{n+1+\alpha}}\right)\left(1-|w|^2\right)^{n+1+\alpha+k}\left|\lim_{r\rightarrow 1}\int_{\mathbb B^n}\frac{b(z)}{(1-\langle w,rz \rangle)^{n+1+\alpha+k}}d\nu_\alpha(z)\right|\\ \le C||h_b||.
\end{multline*}
That is, for all $w\in
\mathbb B^n,$  $$|M_{k}b(w)|\le
C||h_b||(1-|w|^2)^{-k}\rho\left(\left(1-|w|^2\right)^{n+1+\alpha}\right),\,\,$$
Thus $||b||_{\Gamma_{\alpha,\rho}(\mathbb B^n)}\leq C ||h_b||.$ This completes the proof of the theorem.
\end{proof}

\subsection{Boundedness of $h_b$: $\mathcal A_\alpha^{\Phi_1}(\mathbb
 B^n)\rightarrow \mathcal A_\alpha^{\Phi_2}(\mathbb
 B^n)$;  $(\Phi_1,\Phi_2)\in \mathscr{U}^q\times \mathscr{U}^q $}

\begin{theorem}\label{theo:convextoconvex}
Let $\Phi_1$ and $\Phi_2$ in $ \mathscr{U}^q$, and $\rho_i(t)=\frac{1}{t\Phi_i^{-1}(1/t)}$. Denote by $\Psi_2$ the complementary function of $\Phi_2$.
We suppose that:
\begin{itemize}
\item[(i)] $\Phi_2$ satisfies the Dini condition (\ref{dinicondition})
\item[(ii)]  $\frac{\Phi_1^{-1}(t)\Psi_2^{-1}(t)}{t}$ is non-decreasing.
\end{itemize}
 Then   the Hankel operator $h_b$ extends into a bounded operator from $\mathcal A_\alpha^{\Phi_1}(\mathbb
 B^n)$ into $\mathcal A_\alpha^{\Phi_2}(\mathbb
 B^n)$ if and only if its symbol $b$ belongs to  $ \Gamma_{\alpha,\rho_\Phi},$ where
\begin{equation*}
\rho=\rho_{\Phi}:=\frac{\rho_1}{\rho_2}.
\end{equation*}
\end{theorem}
\begin{proof}
Condition $(i)$ implies that the dual space of $\mathcal A_\alpha^{\Phi_2}(\mathbb B^n)$ is $\mathcal A_\alpha^{\Psi_2}(\mathbb
 B^n)$. Condition $(ii)$ implies that $\Phi\in \mathscr{L}_p$ for some $0<p\le 1$. The whole proof follows the lines of the proof of Theorem \ref{theo:concavetoconvex}.
\end{proof}

We observe that in the proof of the above theorem, the condition $(ii)$ is used to ensure that the resulting growth function $\Phi$ is in some $\mathscr{L}_p$. Hence using Lemma \ref{conditionforlowertype}, we have the following.

\begin{proposition}\label{theo:convextoconvex1}
Let $\Phi_1$ and $\Phi_2$ in $ \mathscr{U}^q$, and $\rho_i(t)=\frac{1}{t\Phi_i^{-1}(1/t)}$.
We suppose that
\begin{itemize}
\item[(i)] $\Phi_2$ satisfies the Dini condition (\ref{dinicondition})
\item[(ii)]  $\frac{\Phi_2^{-1}\circ\Phi_1(t)}{t}\quad\textrm{is non-increasing}$.
\end{itemize}
  Then   the Hankel operator $h_b$ extends into a bounded operator from $\mathcal A_\alpha^{\Phi_1}(\mathbb
 B^n)$ into $\mathcal A_\alpha^{\Phi_2}(\mathbb
 B^n)$ if and only if its symbol $b$ belongs to  $ \Gamma_{\alpha,\rho_\Phi}= (\mathcal A_\alpha^{\Phi}(\mathbb
 B^n))^*,$ where
\begin{equation*}
\rho_{\Phi}:=\frac{\rho_1}{\rho_2}.
\end{equation*}
\end{proposition}
\emph{Acknowledgements}: The second author would like to acknowledge the support of the International Centre for Theoretical Physics (ICTP), Trieste (Italy).


\bibliographystyle{plain}

\end{document}